\documentclass[a4paper,11pt]{article}
\usepackage{array}
\usepackage{theorem}
\usepackage{amsmath,amscd,amssymb}
\usepackage{latexsym}
\usepackage{epsfig}
\usepackage{xypic}
\theorembodyfont{\sl}

\newtheorem{lemma}{Lemma}[section]

\newtheorem{proposition}[lemma]{Proposition}
\newtheorem{theorem}[lemma]{Theorem}
\newtheorem{corollary}[lemma]{Corollary}

\newtheorem{definition}[lemma]{Definition}
\newtheorem{remark}[lemma]{Remark}
\newtheorem{question}[lemma]{Question}

\newcommand{\CC}{\mathbb C}

\newcommand{\FF}{\mathbb F}

\newcommand{\PP}{\mathbb P}
\newcommand{\QQ}{\mathbb Q}

\newcommand{\ZZ}{\mathbb Z}

\renewcommand{\cD}{\mathcal D}

\renewcommand{\cH}{\mathcal H}

\renewcommand{\cL}{\mathcal L}
\newcommand{\cM}{\mathcal M}

\newcommand{\cO}{\mathcal O}

\newcommand{\Prod}{\prod\limits}

\renewcommand{\Tilde}{\widetilde}

\newcommand{\Orth}{\mathop{\null\mathrm {O}}\nolimits}

\newcommand{\Rdiv}{\mathop{\mathrm {R.div}}\nolimits}

\newcommand{\sign}{\mathop{\mathrm {sign}}\nolimits}
\newcommand{\torsion}{\mathop{\mathrm {torsion}}\nolimits}

\newcommand{\Hom}{\mathop{\mathrm {Hom}}\nolimits}

\newcommand{\id}{\mathop{\mathrm {id}}\nolimits}

\newcommand{\NS}{\mathop{\null\mathrm {NS}}\nolimits}
\newcommand{\Num}{\mathop{\null\mathrm {Num}}\nolimits}
\newcommand{\Pic}{\mathop{\null\mathrm {Pic}}\nolimits}

\newcommand{\supp}{\mathop{\null\mathrm {Supp}}\nolimits}

\newcommand{\even}{\mathop{\mathrm {ev}}\nolimits}
\newcommand{\odd}{\mathop{\mathrm {odd}}\nolimits}
\newcommand{\Eig}{\mathop{\null\mathrm {Eig}}\nolimits}

\newcommand{\Spec}{\mathop{\null\mathrm {Spec}}\nolimits}

\newcommand{\gothG}{\mathfrak G}

\newcommand{\latt}[1]{{\langle{#1}\rangle}}
\renewcommand{\div}{\mathop{\mathrm {div}}\nolimits}
\newcommand{\Kthree}{\mathop{\mathrm {K3}}\nolimits}
\newcommand{\En}{\mathop{\mathrm {En}}\nolimits}

\newcommand{\qedsymbol}{\mbox{$\Box$}}
\newcommand{\qed}{\unskip\nobreak\hfil\penalty50\hskip1em\hbox{}\nobreak
\hfill\qedsymbol\parfillskip=0pt\finalhyphendemerits=0}
\newenvironment{proof}{\begin{ProofwCaption}{Proof}}{\end{ProofwCaption}}
\newenvironment{ProofwCaption}[1]
 {\addvspace\theorempreskipamount \noindent{\it #1.}\rm}
 {\qed \par \addvspace\theorempostskipamount}

\begin{document}

\title{Moduli of polarized Enriques surfaces}
\author{V.~Gritsenko and K.~Hulek}
\date{}
\maketitle

\section{Introduction}\label{sec:intro}

The moduli space $\cM_{\En}$ of (unpolarized) Enriques surfaces is an open subset of a $10$-dimensional orthogonal modular variety, which was shown by Kond\=o to be 
rational.  In this note we want to discuss moduli spaces of polarized and numerically polarized Enriques surfaces. A polarized Enriques surface is, of course, a pair $(S,\cL)$, where 
$S$ is an Enriques surface and $\cL\in \Pic(S)$ is an ample line bundle. By numerically polarized Enriques surface we mean a pair $(S,h)$ where $h\in \Num(S)$ is the numerical class of an ample
line bundle $\cL$.

One of the main results of this note is Theorem \ref{theo:main}:
for a given polarization $h$, i.e. an $\Orth(U\oplus E_8(-1))$-orbit of a primitive vector of positive degree in the abstract Enriques lattice $U\oplus E_8(-1)$, we construct a suitable orthogonal 
modular variety $\cM_{\En,h}$ of dimension $10$ and identify in this an open subset $\cM_{\En,h}^{a}$ whose points are in $1:1$ correspondence with isomorphism classes of 
numerically polarized Enriques surfaces with this  polarization. Moduli spaces of polarized Enriques surfaces, which exist as quasi-projective varieties by Viehweg's theory, are then given 
by an \'etale $2:1$ cover of $\cM_{\En,h}^{a}$. We ask the question when 
these covers are connected. 

The main conclusion which we derive from our description of the moduli spaces is contained in Corollaries \ref{cor:finitenumerical} and \ref{cor:finitepolarized} and can be stated as follows: 
\begin{theorem}
There are only finitely many isomorphism classes of moduli spaces of polarized and numerically polarized Enriques surfaces.
\end{theorem}

In section \ref{sec:forms} we use automorphic forms to prove that some moduli spaces of numerically polarized Enriques surfaces have negative Kodaira dimension if the corresponding modular  group
contains sufficiently many reflections. For example, this is
true for all polarizations of degree $h^2\le 32$
(see Corollary \ref{corollary:EnriquesExamples}).
In Proposition \ref {proposition-c0} we proved that  
for infinitely many 
polarizations $h$  
the moduli space of numerically $h$-polarised
Enriques surfaces coincides with the moduli space of Enriques 
surfaces with a level-$2$ structure (see \S 2). 
It was announced  in \cite{Gri2} that the last moduli space 
is of general type. 

Moduli spaces of (polarized) Enriques surfaces have been studied by many authors, but not all of the results have appeared in the literature. We have done our best to attribute published results 
wherever possible, but some further results in sections \ref{sec:lattices} to \ref{sec:polarized} are also likely to be known to experts, although they cannot be found in the literature.  
Throughout this note we will be working over the complex numbers $\CC$, but we will also briefly comment on moduli spaces in positive characteristic. 

{\footnotesize 
{\bf Acknowledgement.}  We have profited from many discussions on moduli spaces of Enriques surfaces, in particular with I.~Dolgachev, S.~Kond\=o, E.~Loo\-ijen\-ga, D.~Mar\-ku\-she\-vich 
and S.~Mukai. 
This work was supported by Labex CEMPI in Lille. We are grateful to  G. Nebe and D. Lorch in Aachen for the MAGMA calculations which we used in the proof of 
Corollary \ref{corollary:EnriquesExamples}.
The first author gratefully acknowledges support
by grant of  the Government of the Russian Federation within the 
framework of the implementation of the 5-100 Programme Roadmap of 
the National Research University  Higher School of Economics, AG 
Laboratory.
The second author gratefully acknowledges support by grant DFG Hu 
337/6-2 and the Fund for Mathematics to the  Institute of Advanced 
Study in  Princeton, which provided excellent working conditions.
Finally we thank the referee for his/her careful reading of the article.}

\section{Enriques surfaces and 
the Torelli theorem}\label{sec:lattices}

An Enriques surface (over the complex numbers) is a regular compact complex surface $S$, i.e.\ $q(S)=h^1(S,{\mathcal O}_S)=0$, whose canonical bundle $\omega_S$ is not trivial, but has the property that it is $2$-torsion, i.e. $\omega_S^{\otimes 2}={\mathcal O}_S$. Thus its holomorphic Euler characteristic is 
$\chi(\cO_S)=1$ and  by Noether's formula its Euler number is $e(S)=12$. 

Unlike $\Kthree$ surfaces 
Enriques surfaces are 
always projective \cite[Section V.23]{BHPV} and since 
$H^{0,2}(S)=H^{2,0}(S)=H^0(\omega_S)=0$ 
all classes in $H^2(S,\ZZ)$ are algebraic, 
in particular $H^2(S,\ZZ)\cong \NS(S)$. Since the canonical bundle 
is $2$-torsion the group $H^2(S,\ZZ)$ is not-torsion free. However, 
the canonical bundle is the only torsion 
element and there is a (non-canonical) splitting $H^2(S,\mathbb Z)= 
H^2(S,\mathbb Z)_f \oplus \ZZ / 2\ZZ$ where $H^2(S,\mathbb Z)_f = 
H^2(S,\mathbb Z)/\torsion$ is a free 
module of rank $10$. The cup product, or intersection product, 
endows this with a lattice structure and one has, see \cite[Chapter 
VIII 15.1]{BHPV}
$$
H^2(S,\ZZ)_f=\Num (S) \cong U \oplus E_8(-1)
$$
where $U$ denotes the hyperbolic plane and $E_8(-1)$ is the negative definite $E_8$-lattice, i.e. it is 
negative definite, even, unimodular of rank $8$.

The condition $\omega_S^2=\cO_S$ implies the existence of an
\'etale cover $p: X \to S$ and by surface classification  $X$ is a $\Kthree$ surface. We denote the corresponding involution on $X$ by $\sigma: X \to X$.
For a $\Kthree$ surface $X$ it is well known that the intersection form on $H^2(X,\ZZ)$ is a lattice  of the form
$$
H^2(X,\ZZ)\cong 3U \oplus 2E_8(-1) =: L_{\Kthree}
$$ 
where $L_{\Kthree}$ is the so-called $\Kthree$ lattice. Under the $2:1$ cover $p: X \to S$ the intersection form is multiplied by a factor $2$
and thus 
$$
p^*(H^2(S,\ZZ)) \cong U(2) \oplus E_8(-2)=:M.
$$ 
By \cite[Theorem 1.14.4]{Nik} there is a unique embedding of the lattice  $U(2) \oplus E_8(-2)$ into the $\Kthree$ lattice $L_{\Kthree}$ and thus we may assume that
$M$ is embedded into $L_{\Kthree}$ by the embedding $(x,u) \mapsto (x,0,x,u,u)$. Whenever we refer to the sublattice $M$ of $L_{\Kthree}$ we will use this embedding. 

Consider the involution 
$$
\rho: L_{\Kthree}=3U \oplus 2E_8(-1) \to L_{\Kthree}=3U \oplus 2E_8(-1),
$$
$$
\rho(x,y,z,u,v) = (z,-y,x,v,u).
$$
Clearly $M=\Eig(\rho)^+$ is the $+1$-eigenspace of this involution. Let 
$$
N= U \oplus U(2) \oplus E_8(-2).
$$
We think of $N$ as a primitive sublattice of $L_{\Kthree}$ via the embedding $(y,z,v) \mapsto (z,y,-z,v,-v)$. Clearly
$$
N=M^{\perp}_{L_{\Kthree}}=\Eig(\rho)^-.
$$

Before we  discuss markings and periods we will recall the basics about discriminant forms of lattices.  
For every lattice $L$ its dual is defined by $L^{\vee}= \Hom(L,\ZZ)$, or equivalently 
$L^{\vee}=\{x \in L \otimes \QQ \mid (x,y) \in \ZZ, \, \mbox{for all} \, y \in L\}$.  The discriminant group of $L$ is the finite abelian group 
$$
D_L=L^{\vee}/L.
$$
If $L$ is an even lattice, then the discriminant $D_L$ carries a quadratic form with values in $\QQ/2\ZZ$ induced from the form on $L$. As usual we shall denote the group of isometries
of $L$ and $D_L$ by $\Orth(L)$ and $\Orth(D_L)$ respectively. There is a natural homomorphism $\Orth(L) \to \Orth(D_L)$ and its kernel 
$$
\widetilde \Orth(L)=\{ g\in \Orth(L) \mid g|_{L^{\vee}/L}= \id \}
$$
is called the {\em stable} orthogonal group of $L$. For the lattices $M$ and $N$ we have a natural isomorphism
$$
D_M \cong D_N 
$$
which as an abelian group is the $2$-elementary group $\FF_2^{10}$. Moreover $\Orth(D_M)=\Orth(D_N)\cong \Orth^+(\FF_2^{10})$ is the orthogonal group of {\em even} type,
whose order is $|\Orth^+(\FF_2)|=2^{21}\cdot3^5\cdot5^2\cdot7\cdot17\cdot 31$, for details see \cite[\S 1]{Kon2}, \cite[Chap. I, \S 16, Chap. II. \S 10]{Die}. 
We also know by \cite[Theorem 3.6.3]{Nik} that the homomorphisms $\pi_M: \Orth(M) \to \Orth(D_M)$ and $\pi_N: \Orth(N) \to \Orth(D_N)$ are surjective.

For future use we  also describe a different description of the  group $\Orth(N)$. For this we notice that 
\begin{equation}\label{equ:dualN}
N^{\vee}(2) \cong U \oplus U(2) \oplus E_8(-1) \cong 2U \oplus D_8(-1)
\end{equation} 
and hence
\begin{equation}\label{equ:descriptionN}
\Orth(N) \cong \Orth(N^{\vee}) \cong \Orth(N^{\vee}(2)) \cong \Orth(2U \oplus D_8(-1)).
\end{equation}  

A {\em marking} of an Enriques surface $S$ is an isometry $\varphi: H^2(S,\ZZ)_f \to U \oplus E_8(-1)$.
Every such marking, or more precisely the 
induced marking $\varphi: p^*((H^2(S,\ZZ)) \to U(2) \oplus E_8(-2)$ can be extended (not uniquely) to a 
marking $\widetilde {\varphi}: H^2(X,\ZZ) \to L_{\Kthree}$ of the $\Kthree$-cover $X$. This follows from \cite[Corollary 1.5.2]{Nik} together with the fact that $\Orth(N) \to \Orth(D_N)$
is surjective.
Moreover, we can assume that $\widetilde {\varphi}(H^2(X,\ZZ))=M \subset L_{\Kthree}$, with $M$ the primitive sublattice of $L_{\Kthree}$ as explained above.
The involution $\sigma^*$ acts trivially on $\widetilde {\varphi}(H^2(X,\ZZ))$ and by $-\id$ on its orthogonal complement. This implies that 
$\rho \circ \widetilde {\varphi} = \widetilde {\varphi} \circ \sigma^*$. We shall refer to a marking $\widetilde {\varphi}$ of $X$ with this property as an  {\em Enriques marking}.
Note that if $\widetilde {\varphi}$ and  $\widetilde {\varphi}'$ are two Enriques markings extending the same marking $\varphi$, then 
$$
\widetilde \varphi \circ (\widetilde \varphi')^{-1} |_{L_{\Kthree}^-} \in \widetilde \Orth(N).
$$

Markings allow us to define {\em period points} of Enriques surfaces:
given a marked Enriques surface $(S,\varphi)$ we consider an Enriques marking 
$\widetilde {\varphi}: H^2(X,\ZZ) \to L_{\Kthree}$
as above. The Enriques involution $\sigma$ on $X$ is non-sym\-plec\-tic, i.e. $\sigma^*(\omega_X)=-\omega_X$ and
thus $\widetilde {\varphi}(\omega_X) \in N_{\CC}$.  
The lattice $N$ is an even lattice of signture $(2,10)$ and we can associate to it the type IV domain
$$
\Omega_N=\{ [x] \in \PP(N \otimes \CC) \mid (x,x)=0, \,
(x,\bar{x}) >0 \}
$$
which has two connected components $\Omega_N=\cD_N \cup \cD_N'$. The group $\Orth(N)$ acts properly 
discontinuously on $\Omega_N$ and we denote by $\Orth^+(N)$ the index $2$ subgroup of $\Orth(N)$ with real spinor norm $1$, i.e. the subgroup which
fixes the connected components of $\Omega_N$. 
We will call the group $\Orth^+(N)$ the {\em Enriques modular group}.
Indeed $\Orth^+(N)$  has index $2$ in $\Orth(N)$, since 
the reflection with respect to a $+2$-vector in a hyperbolic plane has real spinor norm $-1$. 
After possibly composing with an isometry $(\id, -\id, \id, \id, \id)$  on $L_{\Kthree}=3U \oplus 2E_8(-1)$, which commutes with $\rho$  and interchanges the two components of $\Omega_N$, 
see \cite[Proposition VIII 20.2]{BHPV},  we may assume 
that $[\widetilde {\varphi}(\omega_X)] \in \cD_N$ is in a fixed connected component.  
For this reason we refer to $\cD_N$ as the {\em period domain of Enriques surfaces}.
Clearly, the period point depends on the choice of the
extension $\widetilde {\varphi}$, but it also depends on the choice of the marking $\phi$ itself. Since every isometry of the sublattice $M \subset L_{\Kthree}$
can be extended to an isometry of the $\Kthree$ lattice $L_{\Kthree}$ one is thus led to consider the action of the group $\Orth(N)$ on $\Omega_N$ , respectively 
$\Orth^+(N)$ on $\cD_N$.

Unlike in the case of $\Kthree$ surfaces not every point in the period domain $\cD_N$ comes from an Enriques surface. 
To describe the image of the period domain we consider all vectors $-2$-vectors $l \in N$, i.e. $l^2=-2$. For each such root we obtain a hyperplane
$$
H_l= \{ [x] \in \cD_N \mid (x,l)=0 \}.
$$
We consider the union
$$
\cH_{-2}= \cup_{l \in N,\, l^2=-2}\, H_l.
$$
It was shown by Horikawa \cite[Main Theorem]{Hor}, see also Nakamura \cite[Theorem 7.2]{Nam} and   \cite[Chapter VIII, 21.4]{BHPV} that the image of the period domain is equal to the set  
$\cD_N \setminus \cH_{-2}$.

Let us now consider the action of the group $\Orth^+(N)$ on $\cD_N$. 
This group acts properly discontinuously and the quotient 
\begin{equation}\label{equ:modular}
\cM_{\En}= \Orth^+(N) \backslash \cD_N
\end{equation}
is a $10$-dimensional quasi-projective variety. It was shown by Namikawa, cf. \cite[Theorem 2.13]{Nam}, that all $-2$-vectors in $N$ are equivalent under 
the action of  $\Orth^+(N)$. Note that, using 
(\ref{equ:descriptionN}),  this can also be deduced by standard 
methods by considering $2U\oplus D_8(-1)$,  
where we remark that $-2$ vectors in $N$ correspond to 
reflective $-4$-vectors in  $2U\oplus D_8(-1)$ with the 
additional property that they pair to an even number with any other 
vector (we will call this later an even $-4$-vector) and vice 
versa.
Hence the union $\cH_{-2}$ maps to an irreducible hypersurface $\Delta_{-2}$ in 
$\cM_{\En}$.   Let 
$$
\cM^0_{\En}= \cM_{\En} \setminus \Delta_{-2}.
$$

The global Torelli theorem for Enriques surfaces as proven by Horikawa \cite{Hor} and refined by Namikawa  in \cite{Nam} implies the following
\begin{theorem}\label{theo:globaltorelli}
There is a bijection
$$
\cM^0_{\En} \stackrel{1:1}{\longleftrightarrow} \{S \mid S \, \mbox{is an Enriques surface} \, \}/ \cong.
$$
\end{theorem}

For this reason the variety $\cM^0_{\En}$ is often referred to as the {\em moduli space of Enriques surfaces}.

\begin{remark}
Strictly speaking it is a misnomer to speak of the moduli space of Enriques surfaces. Although this spaces parametrizes the isomorphism classes of Enriques 
surfaces, it is, at least to us, not known that it represents a moduli functor.
\end{remark}
At this point we would also like to recall the following important theorem due to Kond\=o \cite{Kon1}:

\begin{theorem}[Kond\=o]\label{theo:menrational}
The space $\cM^0_{\En}$ of Enriques surfaces is rational.
\end{theorem}

We also consider the modular variety
\begin{equation}\label{equ:universalmodular}
\widetilde {\cM}_{\En}= \widetilde \Orth^+(N) \backslash \cD_N
\end{equation}
which is a finite cover of $\cM_{\En}$ with Galois group $\Orth^+(\FF_2^{10})$. Its open subset $\widetilde {\cM}^0_{\En}$
which covers $\cM^0_{\En}$ can be interpreted as the moduli space of Enriques surfaces with a level-$2$ structure 
(see \cite[\S 2]{Kon2}).
Gritsenko discussed the modular variety  $\widetilde {\cM}_{\En}=\widetilde \Orth^+(N) \backslash \cD_N$
at the Schiermonnikoog conference  \cite{Gri2}
and outlined a proof that it is of general type, which of course implies that {\it the moduli space Enriques surfaces with a level-$2$ structure  is of general type}.

\section{Moduli spaces of numerically polarized Enriques surfaces}\label{sec:numerical}

In this section we want to describe moduli spaces of numerically polarized Enriques surfaces in terms of modular varieties of orthogonal type.
We first recall that, since  $H^{2,0}(S)=0$ for an Enriques surface $S$, every 
element in $H^2(S,\ZZ)$ is represented by an algebraic class. One consequence of this is that moduli spaces of polarized Enriques surfaces are of 
dimension $10$, in contrast to the situation of $\Kthree$ surfaces, where all $\Kthree$ surfaces form a $20$-dimensional family and 
polarized $\Kthree$ surfaces have dimension $19$.

We have already seen that $\NS(S) = H^2(S,\ZZ) \cong U \oplus E_8(-1) \oplus \ZZ / 2\ZZ$. Since $S$ is regular we can  identify polarizations, i.e. ample line bundles $\mathcal L$ on $S$ with 
their first Chern classes $\tilde h:=c_1(\mathcal L) \in H^2(S,\ZZ)$. A {\em polarized Enriques surface} is a pair $(S,\tilde h)$ where $\tilde h$ represents an ample line bundle 
on $S$. We denote the numerical class defined by $\tilde h$ by $h =[\tilde h] \in \Num(S)= H^2(S,\ZZ)_f=H^2(S,\ZZ) / \torsion$. By a {\em numerically polarized Enriques 
surface} we mean a pair $(S,h)$ where $h \in \Num(S)$ comes from an ample line bundle. Clearly every numerically polarized Enriques surface comes from two polarized Enriques surfaces 
$(S,\tilde h)$ and $(S,\tilde h + K_S)$. Note that e.g. by Reider's theorem $\tilde h$ is ample if and only if $\tilde h + K_S$ is ample.

We shall first discuss moduli of numerically polarized Enriques surfaces. For this we fix a primitive element $h\in U \oplus E_8(-1) = M(1/2)$ of positive degree 
$h^2=2d>0$. 
Note  that again the situation is different from $\Kthree$ surfaces. Any two primitive vectors $h \in L_{\Kthree}$ of the same 
positive degree are equivalent under the orthogonal group 
$\Orth(L_{\Kthree})$. This fails for $h^2 \geq 4$ for the hyperbolic lattice $U \oplus E_8(-1)$.

Now, given $h$, we define the group
$$
\Orth(M(1/2),h)=\Orth(M,h)= \{g \in \Orth(M(1/2))=\Orth(M) \mid g(h)=h \}.
$$
Next, we define the group
\begin{equation}\label{equ:modulargroup}
\Gamma_h =\pi_N^{-1}(\pi_M(\Orth(M,h)))
\end{equation}
where $\pi_M$ and $\pi_N$ are the natural projections onto the finite orthogonal groups $\Orth(D(M))$ and $\Orth(D(N))$ respectively  which, as we have seen, can be identified 
canonically.
Note that $\widetilde \Orth(N) \subset \Gamma_h$ is a normal subgroup of $\Orth(N)$ of finite index and hence $\Gamma_h$ is 
an arithmetic subgroup of $\Orth(N)$. We set
$$
{\widetilde \Orth}^+(N)= \widetilde \Orth(N) \cap \Orth^+(N), \quad \Gamma^+_h = \Gamma_h\cap \Orth^+(N).
$$
Again, we note that both subgroups have index $2$, since the reflection with respect to a vector of length $2$ in the 
summand $U$ of $N$ gives an element in $\widetilde \Orth(N)$ of real spinor norm $-1$.   
The following space is crucial for us
\begin{equation}\label{equ:modulipolarized}
\cM_{\En,h}:= \Gamma^+_h \backslash \cD_N.
\end{equation}
The main point of this section is that we shall show that one can interpret an open set of 
this modular variety as a moduli space of numerically polarized
Enriques surfaces. 
Before discussing this we 
first observe that, if $h$ and $h'$ belong to the same $\Orth(M(1/2))$-orbit, then the groups 
$\Gamma^+_h$ and $\Gamma^+_{h'}$ are conjugate and hence $\cM_{\En,h} \cong \cM_{\En,h'}$. 

Note that for every primitive vector $h\in M$ we have finite covering maps
\begin{equation}\label{equ:tower}
\widetilde {\cM}_{\En} \to \cM_{\En,h} \to {\cM}_{\En}. 
\end{equation}

Recall the hypersurface $\Delta_{-2} \subset {\cM}_{\En}$ and let $\Delta_{-2,h}$ and $\widetilde {\Delta}_{-2}$ be the  pre-images of $\Delta_{-2}$ in  $\cM_{\En,h}$
and $\widetilde {\cM}_{\En}$ respectively. We set
$$
\cM^0_{\En,h}= \cM_{\En,h} \setminus \Delta_{-2,h}, \quad  \widetilde {\cM}^0_{\En} = \widetilde {\cM}_{\En} \setminus \widetilde {\Delta}_{-2}.
$$

We will show that a suitable open subset of $\cM^0_{\En,h}$ gives a moduli space of numerically polarized Enriques surfaces. Before we can 
do this, we need to recall some facts about smooth rational curves and ample divisors on Enriques surfaces. By the adjunction formula a smooth rational 
curve $C$ has self-intersection $C^2=-2$. For this reason we also refer to smooth rational curves as {\em nodal} curves.

Assume that an Enriques surface $S$ contains a smooth rational curve $C$. Then the pre-image of $C$ under the \'etale cover $p: X \to S$
is a union of two disjoint smooth rational curves $C'$ and $C''$ which are interchanged by the involution $\sigma$ and hence the class $[C'] - [C''] \in H^2(X,\ZZ)$ is in the 
$-1$-eigenspace of $\sigma^*$. Hence, given an Enriques marking  $\widetilde {\varphi}: H^2(X,\ZZ) \to L_{\Kthree}$ this defines a primitive vector $l =\widetilde {\varphi}([C'] - [C'']) \in N$ of length
$l^2=-4$.  According to \cite[Theorem 2.15]{Nam} there are two $\Orth(N)$-orbits of vectors of length $-4$ in $N$, and the same 
argument also shows this to be true for $\Orth^+(N)$ . They are distinguished by the divisor $\div(l)$,
which is defined as the positive generator of the ideal $\{(l,n) \mid n \in N \} \subset \ZZ$ and which can be either $1$ or $2$ here. 
If $(u_1,v_1)$ and $(u_2,v_2)$ are a standard basis of the summands $U$ and $U(2)$ of $N$ respectively,
then the two orbits can be represented by $l_{\odd}=u_1 - 2 v_1$ and $l_{\even}=u_2 - v_2$ whose divisors are $1$ and $2$ respectively and which, for this reason, are called {\em odd} 
and {\em even}. Note that their complements in $N$ are $l_{\odd}^{\perp}= E_8(-2) \oplus U(2) \oplus \langle 4 \rangle$ and 
$l_{\even}^{\perp}= E_8(-2) \oplus U \oplus \langle 4 \rangle$. 
It follows from \cite[Proposition 2.16]{Nam} that $l$ is of even type.
As before we consider $H_l= \{ [x] \in \cD_N \mid (x,l)=0 \}$ and the union
$$
\cH_{-4,\even}= \cup_{l^2=-4,\, l \operatorname{even}}H_l.
$$
Since the group $\Orth(N)$ acts transitively on all primitive vectors $l$ of length $l^2=-4$ of given type (\cite[Theorem 2.15]{Nam}), the collection of hyperplanes $\cH_{-4,\even}$  maps to an irreducible hypersurface $\Delta_{-4,\even}$ in  $\cM_{\En}$. Note again that the irreducibility can  also be deduced via the identification (\ref{equ:descriptionN}).

By Namikawa \cite[Proposition 6.2]{Nam} the points in $\cM^0_{\En} \cap \Delta_{-4,\even}$ parameterize those Enriques surfaces which contain a nodal curve, the so called {\em nodal} Enriques surfaces. Thus the  open set
\begin{equation}\label{equ:nonodal}
\cM_{\En}^{nn}= \cM_{\En} \setminus (\Delta_{-2} \cup \Delta_{-4,\even}) 
\end{equation}
parametrizes the {\em non-nodal}Ê Enriques surfaces.

The following lemma is standard, but we recall it for the reader's convenience.

\begin{lemma}\label{lem:ample}
Let $\cL$ be a nef line bundle on an Enriques surface $S$ with $c_1(\cL)^2 > 0$. Then $\cL$ is ample if and only if there is no nodal curve $C$ 
with $c_1(\cL).C=0$.
\end{lemma}
\begin{proof}
Since $c_1^2(\cL)>0$ it follows from Riemann-Roch that $\cL$ or $\cL^{-1}$ must be effective. Since $\cL$ is nef it must be $\cL$ itself. 
To show ampleness it is enough by the Nakai-Moishezon criterion to show that $c_1(\cL).C>0$ for every irreducible curve $C$.
Again by nefness of $\cL$ the only obstruction to ampleness can thus be an irreducible curve $C$ with $c_1(\cL).C = 0$. But then $ C \in c_1(\cL)^{\perp}$ and  the orthogonal 
complement of any vector of positive degree in the hyperbolic lattice  $\Num(S) \cong U \oplus E_8(-1)$ is negative definite, which implies $C^2<0$, which in turn means that $C$ is a nodal curve
orthogonal to $c_1(\cL)$.
\end{proof}

Recall that we have fixed a primitive vector $h \in M \subset L_{\Kthree}$ where $M$ is the $+1$-eigenspace of the involution $\rho$. We fix the following set of roots in the $\Kthree$ lattice
\begin{equation}
R_h= \{\delta \in L_{\Kthree} \mid \delta^2=  -2, \delta.\rho(\delta)=0, \delta.h=0\}.
\end{equation}
Note that, since $h$ is invariant under $\rho$, the condition $\delta.h=0$ is equivalent to $(\delta + \rho(\delta)).h=0$ and implies $(\delta - \rho(\delta)).h=0$. Also note that  $\delta - \rho(\delta) \in N$
is an even vector of length $-4$ \cite[Proposition 2.16]{Nam}. By \cite[Proposition 4.5]{Nam} the interpretation of $R_h$ is the following: the classes $\delta$ and $\rho(\delta)$ correspond to 
the classes $[C']$ and $[C'']$ where $C$ is a nodal curve on $S$ and $p^{-1}(C)=C' +C''$ with the additional property that $h.p^{-1}(C)=0$.
As before we set $H_{\delta - \rho(\delta)}=\{x \in \cD_N \mid (x,\delta - \rho(\delta))=0 \}$.
Let 
\begin{equation}
\cH_{R_h}= \cup_{\delta \in R_h}H_{\delta - \rho(\delta)}. 
\end{equation}
It follows from the construction of $\Gamma_h$ and \cite[Corollary 1.5.2]{Nik} that every automorphism in $\Gamma_h$ can be extended to an isometry of the $\Kthree$ lattice $L_{\Kthree}$ in such a way that it 
fixes $h$ and commutes with the involution $\rho$. This implies that $R_h$ is mapped under $\Gamma_h$ to itself and $\cH_{R_h}$ maps to a hypersurface $\Delta_{-4,\even, h^{\perp}}$ in $\cM_{\En,h}$. 
Note that if $\Delta_{-4,\even,h}$ denotes the pre-image of $\Delta_{-4,\even}$ under the map $\cM_{\En,h} \to \cM_{\En}$, then by construction $\Delta_{-4,\even, h^{\perp}} \subset \Delta_{-4,\even,h}$. 
Geometrically this is the obvious fact that non-nodal Enriques cannot contain nodal curves orthogonal to $h$. 

 Finally we set 
\begin{equation}\label{equ:amplelocus}
\cM_{\En,h}^{a}= \cM_{\En,h} \setminus (\Delta_{-2,h} \cup \Delta_{-4,\even, h^{\perp}}). 
\end{equation}
By what we have just  said $\cM_{\En,h}^{nn} \subset \cM_{\En,h}^{a}$ is an open set.
The main result of this section is the following:
\begin{theorem}\label{theo:main}
The open subset $\cM^a_{\En,h}$ of $\cM^0_{\En,h}$ has the following property: its points 
are in $1:1$-correspondence with isomorphism classes of numerically polarized Enriques surfaces $(S,h)$. 
\end{theorem}
\begin{proof}
We start with a pair $(S,H)$ and choose a {\em marked polarization}, i.e. a polarization 
$\varphi: H^2(S,\ZZ)_f \to M(1/2)$ with $\varphi(H)=h$, which we extend to an Enriques marking on the 
$\Kthree$ double cover
$\widetilde{\varphi}: H^2(X,\ZZ) \to L_{\Kthree}$. We then associate to $(S,H)$ the class of the period point 
$\widetilde{\varphi}(\omega_X)$ in $\cM^0_{\En,h}$. We must first show that this is well defined.
Two different  Enriques markings extending $\varphi$ differ by an element in $\widetilde{\Orth}(N)$. As this
is a subgroup of $\Gamma_h$ this defines the same point in $\cM^0_{\En,h}$. Next we have to consider the case where 
we have a different polarized marking $\varphi': H^2(S,\ZZ)_f \to M(1/2)$ with $\varphi'(H)=h$.
Let $\widetilde{\varphi}$ and $\widetilde{\varphi}'$ be Enriques markings extending $\varphi$ and $\varphi'$. Then 
$\widetilde{\varphi} \circ \widetilde{\varphi}'^{-1}|_N \in \Gamma_h$ by the definition of 
the group $\Gamma_h$, and thus the map is well defined. 

Clearly this map send $(S,H)$ and $(S,H+K_S)$ to the same point in $\cM^a_{\En,h}$. Next we want to show that 
these are the only points which are identified. 
Let $(S,H)$ and $(S',H')$ be two polarized Enriques surfaces defining the same point in
$\cM^a_{\En,h}$. We want to show that $(S',H')\cong (S,H)$ or $(S',H')\cong (S,H+K_S)$. For this we consider
the $\Kthree$ covers $(X,H)$ and $(X',H')$ together with polarized Enriques markings $\widetilde{\varphi}$ and 
$\widetilde{\varphi}'$ 
respectively. Let $\psi\in \Gamma_h$ be an automorphism with
$\psi(\widetilde{\varphi}(\omega_X))=\widetilde{\varphi}'(\omega_X')$. By definition of the group $\Gamma_h$ we can
extend $\psi$ to an isometry $\widetilde\psi \in \Orth(L_{\Kthree})$ with $\widetilde\psi(h)=h$. 
Since $\widetilde\psi$ respects the subspaces $M$ and $N$ it follows that it commutes with $\rho$. Hence
$\eta=(\widetilde{\varphi}')^{-1} \circ \widetilde\psi \circ \widetilde{\varphi}: H^2(X,\ZZ) \to  H^2(X',\ZZ)$
is a Hodge isometry with the additional properties that it commutes with the Enriques involutions, i.e. 
$\eta \circ {\sigma}^* =(\sigma')^* \circ \eta$ and that $\eta(H)=H'$. Since $H$ and $H'$ are ample we can
apply the strong Torelli theorem to conclude that there is an isomorphism $f: X' \cong X$ with 
$\eta=f^*$ and $f^*(H)=H'$.
Since moreover $f$ commutes with the Enriques involutions on $X$ and $X'$ it descends to an isomorphism 
$g: S' \to S$ with $g^*(H)=H'$ or $g^*(H)=H'+K_{S'}$.  

It remains to prove that every point in $\cM_{\En,h}^{a}$ comes from a polarized Enriques surface. By the surjectivity of the period map for $\Kthree$-surfaces we can assume that there
is a pair$(X,\cL)$ where $X$ is a $\Kthree$-surface and $\cL$ a semi-ample line bundle, together with a marking
$\widetilde {\varphi}: H^2(X,\ZZ) \to L_{\Kthree}$ such that $\widetilde {\varphi}(\omega_X)=\omega$. 
Via the marking $\widetilde {\varphi}$ the involution $\rho$ induces an involution $\iota$ on $H^2(X,\ZZ)$ which is a Hodge isolmetry. We want to argue that $\iota=\sigma^*$  
for an Enriques involution  $\sigma: X \to X$ which has the additional property that it fixes $c_1(\cL)$. For this we argue similar to 
the proof of \cite[Theorem 3.13]{Nam}. The idea is to find an element  $w \in W_X$ in the Weyl group of $X$ such that $w \circ \iota \circ w^{-1}$ is
an effective Hodge isometry. The main point is to find such a $w$ with the additional property that $w(c_1(\cL))=c_1(\cL)$. That this can be done 
follows from the fact that the subgroup $W_{X.h}$ of $W_X$ generated by reflection by roots orthogonal to $h$ acts transitively on the 
chambers of the positive cone, see \cite[p. 151]{Beau}. But then we can argue as in the proof of \cite[Theorem 7.2]{Nam}: the involution $w \circ \iota \circ w^{-1}$ 
is induced by an involution on $X$ which can be shown to have no fixed points, i.e. is an Enriques involution. Hence the quotient of 
$(X,\cL)$ is a  pair $(S,\cM)$ where $S$ is an Enriques surface and $\cM$ is a nef line bundle. (The involution $\iota$ can be lifted in two ways to an 
involution of the line bundle $\cL$, whose quotients give rise to $\cM$ and $\cM \otimes \omega_S$ respectively). The fact that $\cM$ is ample 
now follows from Lemma \ref{lem:ample} since $\omega \notin \cH_{R_h}$ implies that there are no nodal curves on which $\cM$ has degree $0$.
\end{proof}

\begin{remark}
The points on the hypersurface $\Delta_{-4,\even, h^{\perp}} \setminus \Delta_{-2}$ are in $1:1$ correspondence with numerically {\em semi-polarized} Enriques surfaces, where semi-polarization as 
usual means that the line bundle is nef but not ample. The argument is as in \cite[Section 5]{HP}, see also \cite{Beau}.
\end{remark}

\begin{remark}
Note that the variety $ \cM_{\En,h}$ and the hypersuface $\Delta_{-4,\even, h}$ contained in it  only depend on the finite subgroup $\pi_M(\Orth(M,h))$ in $\Orth(D(M))$.
The hypersurface $\Delta_{-4,\even, h^{\perp}}$ on the other hand  a priori depends on $h$ itself. The difference between  $\Delta_{-4,\even, h^{\perp}}$  and $\Delta_{-4,\even, h}$ 
is that $\Delta_{-4,\even, h^{\perp}}$ contains only  some of the components of $\Delta_{-4,\even, h}$.  
\end{remark}

One corollary from this is the following finiteness result:
\begin{corollary}\label{cor:finitenumerical}
There are only finitely many different birational and isomorphism classes of moduli spaces of numerically polarized Enriques surfaces. 
\end{corollary}
\begin{proof}
By Theorem \ref{theo:main} every such moduli space is birational  to a variety $\cM_{\En,h}$,
which in turn only depends on a subgroup in $\Orth(D(M))$. Since this is a finite group the result follows.  However we can say more. Since  $\Delta_{-4,\even, h}$ only has finitely many components
there are only finitely many possibilities for moduli spaces of polarized Enriques surfaces $\cM^{nn}_{\En,h} \subset  \cM^a_{\En,h} \subset \cM^0_{\En,h}$ and thus we also obtain the statement 
about the isomorphism classes. 
\end{proof}

\section{Moduli spaces of polarized Enriques surfaces}\label{sec:polarized}
In this section we want to discuss moduli spaces of polarized Enriques surfaces, i.e. pairs $(S,\cL)$ where $S$ is an Enriques surfaces and $\cL$ is an ample line bundle.
We fix the $\Orth(M)$-orbit  of the numerical polarization defined by $c_1(\cL)$. By Viehweg's theory \cite[Theorem 1.13]{Vie} there exists a quasi-projective moduli space $\widehat{\cM}^a_{\En,h}$ 
for these pairs.
\begin{proposition}\label{prop:map}
There exists an \'etale $2:1$ morphism $\widehat{\cM}^a_{\En,h} \to \cM^{{}a}_{\En,h}$.
\end{proposition}
\begin{proof}
We use the map which maps $(S,\cL)$ and $(S,\cL\otimes \omega_S)$ to the numerically polarized Enriques surface$(S,h)$ where $h$ is the 
class of $c_1(\cL)$ in $H^2(S,\ZZ)_f$.
Arguing as in \cite[Theorem 1.5]{GHS2}, using 
Borel's extension map \cite{Bo}, we find that this map is not only a holomorphic map, but a morphism of quasi-projetive varieties, see also \cite[Proposition 2.2.2]{Has}.  
\end{proof}

It is not  clear whether the covering $\widehat{\cM}^a_{\En,h} \to \cM^{{}a}_{\En,h}$ is connected or not. 
The answer is known to be  positive in some cases. Classically studied examples include the polarizations in degree $4$ and $6$ where the polarization is base point free.
Note that in both degrees we have two orbits of primitive vectors. One case is given by $h=e+df \in U$, $d=2,3$. In this case $h^{\perp} \cong \langle -2d \rangle \oplus E_8(-1)$.
The corresponding polarizations $H$ are in general ample but never base point free (and are in the literature partly excluded as polarizations on Enriques surfaces). The reason they are not base point free is
that $|2f|$ defines an elliptic fibration with two double fibres $F$ and $F'$ (differing by the canoncal bundle) and $H.F=H.F'=1$. For the other cases 
$h^{\perp} \cong D_9(-1)$ and $h^{\perp} \cong A_2(-1) \oplus E_7(-1)$ respectively. The first of these cases was treated by Casnati who proved connectedness of the moduli space and 
rationality  in  \cite{Cas}. In \cite{Lie} these polarizations are called {\em Cossec-Verra polarizations}.
The second case is simply the classical fact that
a general Enriques surface can be realized as a sextic surface in ${\mathbb P}^3$ passing doubly through the edges of the coordinate tetrahedron. 
This space is also known to be rational, see  Dolgachev \cite{Dol1}, \cite{Dol2}. For further discussions about polarized Enriques surfaces, in particular of degrees $2 \leq h^2 \leq 10$, 
we refer the reader to \cite{Dol3}.

\begin{question}
When is the degree $2$ cover $\widehat{\cM}^a_{\En,h} \to \cM^{{}a}_{\En,h}$ connected?
\end{question}
This question is related to the notion of {\em supermarked} Enriques surfaces, which has been developed by Dolgachev and Markushevich. 
A {\em su\-per\-mar\-king} is an isometry $\varphi: H^2(S,\ZZ) \to U \oplus E_8(-1) \oplus \ZZ/2\ZZ$.

A further question, which we do not know the answer to, is the following:
\begin{question}
Is $\widehat{\cM}^a_{\En,h}$ the quotient of $\cD_N$ by a suitable arithmetic group?
\end{question}
Of course a positive answer to that would imply that $\widehat{\cM}^a_{\En,h}$ is connected.

The above description is, however, enough to prove the 
\begin{corollary}\label{cor:finitepolarized}
There exist only finitely many different isomorphism classes of moduli spaces $\widehat{\cM}^a_{\En,h}$ of polarized Enriques surfaces
\end{corollary}
\begin{proof}
It is enough to show that each variety $\cM^a_{\En,h}$ only admits finitely many \'etale $2:1$ coverings. This follows since $\cM^a_{\En,h}$ is a finite CW complex whose 
degree $2$ coverings are classified by the elements in $H^1(\cM^a_{\En,h},\ZZ/2 \ZZ)$. This is a finite group.
\end{proof}

At this point we would like to comment briefly on the work of Liedtke in \cite{Lie}. There he considers the moduli problem also in positive characteristic. He treats in particular the case of 
Cossec-Verra polarizations in detail, see \cite[Theorem 4.12]{Lie}. Liedtke shows that this moduli problem carries the structure  of a
quasi-separated Artin stack of finite type over $\Spec \ZZ$, which over $\ZZ[\frac12]$ is even a Deligne-Mumford stack. In characteristic $p >2$  the stack is irreducible and smooth of dimension $10$.
Liedtke also considers the functor of unpolarized Enriques surfaces. This stack is, however, badly behaved, in particular its diagonal is not quasi-compact, see \cite[Remark 5.3]{Lie}. It is not clear whether this 
stack has an underlying coarse moduli space and if, how this is related to $ {\cM}_{\En}$ (over the complex numbers).

\section{Modular varieties  of negative Kodaira dimension}\label{sec:forms}
\label{sec:mod-varieties}

In this section we describe a class of modular varieties of dimension $10$ of negative Kodaira dimension.

\begin{theorem}\label{theo:Gamma}
Let $\Gamma^+$ be a  group  between the Enriques modular group 
and its stable subgroup
$\widetilde\Orth^+(N)< \Gamma^+ < \Orth^+(N)$.
We assume that $\Gamma^+$ contains at least $26$ reflections
which are  not conjugate 
with respect to $\widetilde\Orth^+(N)$. Then the Kodaira dimension of  the modular variety   
$
\cM_{\Gamma^+} = \Gamma^+ \backslash \cD_N
$
is negative.
\end{theorem}
To prove this theorem we use the theory of reflective modular forms 
together with the general results about the compact models
of modular varieties of orthogonal type obtained in 
\cite{GHS1}. 
\begin{definition}\label{def-mf} Let $\sign(L)=(2,n)$ with $n\ge 3$
and let $\Gamma^+$ be a subgroup of $\Orth^+(L)$ of finite index.
A  modular form of weight $k$ and character $\chi\colon \Gamma^+\to
\CC^*$ with respect to $\Gamma^+$  is a holomorphic function $F\colon\cD(L)^\bullet\to \CC$
on the affine cone $\cD_L^\bullet$ over $\cD_L$ such that
\begin{align*}
F(tZ)&=t^{-k}F(Z)\quad \forall\,t\in \CC^*,\\
F(gZ)&=\chi(g)F(Z)\quad  \forall\,g\in \Gamma^+.
\end{align*}
\end{definition}
Note that by Koecher's principle these forms are automatically holomorphic at the boundary. We denote the (finite dimensional) space of 
modular forms of weight $k$ and character $\chi$ with respect to the group $\Gamma^+$ by $M_k(\Gamma^+, \chi)$ and the space of 
{\em cusp forms}, i.e. those forms vanishing at the boundary of the Baily-Borel compactification, by  $S_k(\Gamma^+, \chi)$.

The geometric type of a modular variety of orthogonal type
depends very much on its ramification divisor.
For any  non isotropic vector $r\in L$  we denote by $\sigma_r$ the  reflection with respect to $r$:
$$
\sigma_r(l)=l-\frac{2(l,r)}{(r,r)}r\in  \Orth(L\otimes \QQ).
$$
This is an element of  $\Orth^+(L\otimes \QQ)$ if and only if 
$(r,r)<0$ ($\sign(L)=(2,n)$).
A vector $r$ is called {\it reflective} if 
$\sigma_r\in {\Orth}^+(L)$.
If $r^2=-2$, then  $\sigma_r\in \Tilde{\Orth}^+(L)$.

\begin{definition}\label{def-reflf}
A modular form $F\in M_k(\Gamma^+,\chi)$ is called {\em reflective} if
\begin{equation*}\label{eq-reflf}
\supp (\div F) \subset \bigcup_{\substack{\pm r\in L \vspace{1\jot}\\
r\  {\rm is\   primitive}\vspace{1\jot}\\
\sigma_r\in \Gamma^+\text{ or }-\sigma_r\in \Gamma^+}} H_r(L)
=:\Rdiv(\pi_{\Gamma^+})\subset \cD_L.
\end{equation*}
We note that $\sigma_{r}=\sigma_{-r}$ and $H_r=H_{-r}$.
\end{definition}

The divisor $\Rdiv(\pi_{\Gamma^+})$ in the above definition is
the {\em ramification divisor} of the modular projection
$\pi_{\Gamma^+}: \cD_L\to \Gamma^+\setminus \cD_L$ according to 
\cite[Corollary 2.13]{GHS1}.
The ramification divisor of the full orthogonal group 
$\Orth^+(N)$ has two irreducible components 
$\Delta_{-2}$ and  $\Delta_{-4,\even}$
defined by  $-2$- and $-4$-reflective vectors in $N$
(see \S 3 above).
We need two reflective
modular forms, the so-called automorphic discriminants,
of this moduli space.

\begin{lemma}\label{lemma:Reflective} There exist two reflective 
modular forms 
\begin{align*}
\Phi_4&\in M_4(\Orth^+(N), \chi_2),\qquad
\ \div_{\cD_N} \Phi_4= \cH_{-2},\\
\Phi_{124}&\in S_{124}(\Orth^+(N), \mu_2),\qquad
\div_{\cD_N}\Phi_{124}= \cH_{-4,ev}
\end{align*}
where $\chi_2$ and $\mu_2$ are two binary characters of $\Orth^+(N)$.
Both modular forms vanish along the corresponding divisor 
with order one.
\end{lemma} 

\begin{proof}  This result is, in principle, known.
The form $\Phi_{4}$ was found  in \cite{Bor} and  is now called 
the Borcherds--Enriques modular form (of weight $4$).
The additive lifting construction of $\Phi_{4}$ in terms of 
vector valued modular forms was proposed by Kond\=o
in \cite[Proposition 4.6]{Kon2}.
For the second function see \cite[Theorem 4.4]{Kon2}.
Note however, that the  modular groups in the original 
constructions were smaller.  
We propose here simple constructions of these forms
which give us the maximal modular group, the formula for the characters together with the fact that the second form is cusp. 
Recall from (\ref{equ:descriptionN}) that
\begin{equation*}
\Orth^+(N) \cong \Orth^+(2U\oplus D_8(-1))
\end{equation*}
and that under the transformation $N^\vee \to N^\vee(2)$ the   $-2$-vectors of $N$ and  the even $-4$-reflective vectors 
(or the $-1$-reflective vectors in the dual lattice)
transform to the $-4$-reflective vectors and the $-2$-vectors respectively in $2U\oplus D_8(-1)$.

We recall that $D_8$ is an even  sublattice  of the Euclidian lattice $\ZZ^8$
$$
D_8=\{n_1e_1+\ldots+n_8e_8\,|\, n_i\in \ZZ,\ n_1+\dots+n_8\in 2\ZZ\}.
$$
In \cite{Gri1} the modular form $\Phi_4$ was constructed as the 
Jacobi lifting of the product of eight Jacobi theta-series
$$
\Phi_4=\hbox{Lift}(\vartheta(\tau,z_1)\cdot \ldots 
\cdot\vartheta(\tau,z_8))
\in M_4(\Orth^+(2U\oplus D_8(-1)), \chi_2)
$$
where $z_1e_1+\ldots+z_8e_8\in D_8\otimes \CC$.
We note that 
$$
\Orth^+(2U\oplus D_8(-1))=
\Tilde\Orth^+(2U\oplus D_8(-1))\cup \sigma_{-4}
\Tilde\Orth^+(2U\oplus D_8(-1))
$$
where  $\sigma_{-4}$ is the reflection with respect to any 
$-4$ reflective vector. (For example, one can take the transformation
$e_1\mapsto -e_1$.)
Using this we see that 
$\chi_2|_{\Tilde O^+(2U\oplus D_8(-1))}=1$ and $\chi_2(\sigma_{-4})=-1$.

To construct the second automorphic disriminant vanishing along 
the  divisors defined by all $-2$-vectors of $2U\oplus D_8(-1)$
we consider
this lattice as a primitive sublattice of the even unimodular 
lattice  $2U\oplus N(3D_8(-1))$ of signature $(2,26)$ 
where $N(3D_8)$ is the unimodular  Niemeier lattice with 
root lattice $3D_8$. The arguments identical to the considerations
in Theorem 3.2 of \cite{GH2} show that the quasi-pullback 
(see \cite[\S 8]{GHS3}) of the Borcherds form 
$\Phi_{12}\in M_{12}(\Orth^+(2U\oplus N(3D_8(-1))),\det)$
to the  lattice $2U\oplus D_8(-1)$ is a $-2$-reflective 
cusp  form of weight $124$:
$$
\Phi_{124}\in S_{124}(\Orth^+(2U\oplus D_8(-1)), \mu_2)
$$
where $\mu_2|_{\Tilde\Orth^+(2U\oplus D_8(-1))}=\det$,
and $\mu_2(\sigma_{-4})=1$. We note that 
$\mu_2(\sigma_{-2})=\det(\sigma_{-2})=-1$
because $\sigma_{-2}\in \Tilde\Orth^+(2U\oplus D_8(-1))$.
\end{proof}

\noindent
{\it{Proof of Theorem \ref{theo:Gamma}}}.
The modular variety $\cM_{\Gamma^+}$  of dimension $10$ has 
a projective toroidal compactification $\overline{\cM}_{\Gamma^+}$ 
with canonical singularities and no ramification
divisors at infinity (see \cite[Theorem 2]{GHS1}).
To prove the theorem  we have to show that there are no
pluricanonical differential forms on $\overline{\cM}_{\Gamma^+}$.
Suppose that $F_{10k}\in M_{10k}(\Gamma^+, \det^k)$. 
We may realise $\cD_N$ as a tube domain  by choosing a 
$0$-dimensional cusp.
In the corresponding affine coordinates of this tube domain
we take  a holomorphic volume element $dZ$ on $\cD_N$. Then
the differential form $\Omega(F_{10k})=F_{10k}\,(dZ)^k$ 
is $\Gamma^+$-invariant.
Therefore it determines a section of the pluricanonical bundle
over a smooth open part of  the modular variety away from the branch divisor and the boundary. 
Assume that $\Omega(F_{10k})$ can be extended 
to a global section 
$H(\overline{\cM}_{\Gamma^+,} kK_{\overline{\cM}_{\Gamma^+}})$.
It follows that the modular form  $F_{10k}$ vanishes with order $k$
along the ramification divisor of  $\Gamma^+$ in $\cD_N$.

The group $\Gamma^+$ contains the element $-\hbox{id}_N$.
According to \cite[Corollary 2.13]{GHS1} 
the ramification divisor of $\pi_\Gamma^+$ is equal to 
\begin{equation*}
\bigcup_{\substack{\pm r\in N,\ r^2=-2,\ -4
 \vspace{1\jot}\\
\sigma_r\in \Gamma^+}} H_r.
\end{equation*}
The ramification divisor always contains  components $H_r$ with  
$-2$-vectors $r$  because the stable orthogonal group contains 
all such reflections:
$\sigma_r\in \Tilde\Orth^+(N)<\Gamma^+$.
Moreover  all $-2$-vectors of the lattice  $N$ 
belong to the same $\Tilde\Orth^+(N)$-orbit. 
Therefore $F_{10k}$  vanishes along $H_r$ ($r^2=-2$) with order $k$
 and it is divisible 
by the $k$-power of the Borcherds-Enriques modular form $\Phi_4$.
According to the Koecher principle we have
\begin{equation*}
F_{6k}=\frac{F_{10k}}{\Phi_4^k}\in M_{6k}(\Gamma^+,\chi)
\end{equation*} 
where $\chi$ is a binary character of $\Gamma^+$.
The modular form $F_{6k}$ vanishes with order $k$
along the ramification divisor of  $\Gamma^+$ associated with all even -4-vectors.

Starting from $F_{6k}$ we can construct a modular form with respect 
to 
$\Orth^+(N)$ with $-4$-reflective divisor using the method of multiplicative symmetrisation
(see \cite[\S 3.2]{GN} and \cite[\S 1]{GH1}). 
We put
\begin{equation*}
F^{Sym}(Z)=\Prod_{g\in \Gamma^+ \setminus \Orth^+(N)} F_{6k}(gZ)
\in M_{6k[\Orth^+(N):\Gamma^+]}(\Orth^+(N), \chi')
\end{equation*}
where $\chi'$ is a character of  $\Orth^+(N)$. 
We note that the function 
$F^g_{6k}(Z)=F_{6k}(gZ)$ is a modular form with respect to the group 
$g^{-1}\Gamma^+ g$ 
containing  the normal subgroup
$\Tilde\Orth^+(N)$.
The modular form  $F_{6k}$ vanishes with order $k$ along the $-4$-reflective divisors $H_r$ where $r^2=-4$ and $\sigma_r\in \Gamma^+$. Therefore  $F^g_{6k}$ vanished along $H_{g^{-1}r}$.
The $-4$-part $\Delta_{-4,ev}$ of the ramification divisor 
of $\Orth^+(N)\setminus \cD_N$ is irreducible because 
all $-4$-reflective vectors belong to the same $\Orth^+(N)$-orbit.
Two $-4$-reflective vectors $r_1$ and $r_2$ belong to the same 
$\Tilde\Orth^+(N)$-orbit if and only if they have the same  images in the discriminate group or equivalently if  
$\frac{r_1}2\equiv \frac{r_2}2\mod N$ (see \cite[Corollary 3.3]{Ste}).
We note that $\frac{r_1}2 \mod N$ is a non isotropic element
in the discriminant group $D(N)\cong (\mathbb F_2^{10}, q^+)$ of 
the lattice $N$.
This quadratic space has $496$ non isotropic vectors
and all of them might be obtained as the image of a $-4$-reflective
vector in $U(2)\oplus E_8(-2)$ (see \cite[\S 9]{CD}). 
Therefore there exist  $496$ different $\Tilde\Orth^+(N)$-orbits of $-4$-reflective vectors
in $N$.

Let $R$ be the number of  $\Tilde\Orth^+(N)$-orbits 
of $-4$-reflections in $\Gamma^+$. The multiplicity of 
$F^{Sym}$ along the irreducible divisor $\Delta_{-4, ev}$
of the modular variety $\Orth^+(N)\setminus \cD_N$ is equal to 
$$
m=\frac{kR[\Orth^+(N):\Gamma^+]}{496}.
$$
Therefore $F^{Sym}$ is divisible by the $m$-th power
of the  reflective form 
$\Phi_{124}$ and 
$$
6k[\Orth^+(N):\Gamma^+]\ge 124m= \frac{kR[\Orth^+(N):\Gamma^+]}{4}.
$$
Thus  $24\ge R$. It follows that 
$H(\overline{\cM}_{\Gamma^+}, kK_{\overline{\cM}_{\Gamma^+}})$
is trivial for all $k$ if $R\ge 25$. This finishes the proof of
Theorem \ref{theo:Gamma}.
\newline
${}\qquad\qquad\qquad\qquad\qquad \qquad\qquad\qquad\qquad\qquad
\qquad\qquad\qquad\qquad\qquad\quad\ \ \Box$
\smallskip

\begin{corollary}\label{corollary:EnriquesKodaira}
Let $h\in U\oplus E_8(-1)$ be a primitive element such that 
$h^2=2d>0$ and the negative definite lattice 
$h^\perp_{U\oplus E_8(-1)}$ contain more than $48$ vectors 
with length $-2$. Then the moduli space 
$\cM_{\En,h}=\Gamma^+_h \setminus \cD_N$
of $h$-polarized Enriques surfaces has negative Kodaira dimension.
\end{corollary}
\begin{proof} We first note that we have 
$\Tilde\Orth^+(N)<\Gamma^+_h< \Orth^+(N)$.
Secondly, any two $-2$-vectors $u\ne \pm v$ in  
the negative definite even lattice $h^\perp_{U\oplus E_8(-1)}$
are not congruent modulo  
$2h^\perp_{U\oplus E_8(-1)}$ because $|(u-v, u-v)| < 8$.
Therefore they define more than 
$24$ reflections non conjugate with respect to 
$\Tilde\Orth^+(N)$.  We note that 
$\Orth(M,h)\cong \Tilde\Orth(h^\perp_{U\oplus E_8(-1)})$.
We finally remark that $\Gamma^+_h$ contains
the class of $-2$-reflections.
\end{proof}
\begin{corollary}\label{corollary:EnriquesExamples}
The modular variety
$\cM_{\En,h}:= \Gamma^+_h \backslash \cD_N$ 
has negative Kodaira dimension for all polarizations
$h\in U\oplus E_8(-1)$ of degree $h^2\le 32$.
For $h^2=34$, $36$, $38$, $40$, $42$ 
the same is true for all polarizations of the corresponding 
degree except possibly one. 
\end{corollary}
\begin{proof} We take a primitive vector in 
$h_{2d}\in U\oplus E_8(-1)$ such that $h_{2d}^2=2d$.
The hyperbolic lattice is unimodular and 
the discriminant form of the orthogonal complement
$h_{2d}^{\perp}$ is equal to the discriminant form of 
rank one lattice  $\latt{-2d}$. 
Therefore the lattice $h_{2d}^{\perp}$ belongs to the genus 
$\gothG(2d)=\gothG(\latt{-2d}\oplus E_8(-1))$ 
which has  finite number of classes. They   
characterise all primitive vectors $h_{2d}$ of degree $2d$.
For $d=1$ there exist only one class $\latt{-2}\oplus E_8(-1)$.
For $d=2$, $3$, $4$ the number of classes is equal to  $2$ 
and we can determine the second class. For larger $d$ a 
MAGMA computation performed by Gabi Nebe and David Lorch  in Aachen gives the following result where the notation 
${\mathbf d}: [n_1, n_2,\dots, n_k]$  means 
the half-number of roots in the $k$ different classes 
of the genus $\gothG(2d)$.
We always have $n_1=120$ for the ``trivial" class 
$\latt{-2d}\oplus E_8(-1)$.
\begin{align*}
{\mathbf 2}&: [ 120,\  72 ]\ \ {\mathbf 5}: [ 120, 64, 45 ]
\ \  \ \,{\mathbf 8}:[ 120, 64, 56, 36 ]
\ \  {\mathbf 1 \mathbf 1}:[ 120, 64, 63, 36, 33 ]\\
{\mathbf 3}&: [ 120,\  66 ]\ \  {\mathbf 6}:[ 120, 56, 42 ]
\ \ \ \,{\mathbf 9}: [ 120, 64, 39, 37 ]
\ \  {\mathbf 1\mathbf 2}: [ 120, 56, 39,  29 ]\\
{\mathbf 4}&: [ 120,\  56 ]\ \  {\mathbf 7}: [ 120, 64, 43 ]
\ \  {\mathbf 1\mathbf 0}: [ 120, 56, 42, 30 ]
\ \  {\mathbf 1\mathbf 3}: [ 120, 64, 42, 38, 29 ]\\
\mathbf 1\mathbf 4&: [ 120, 63, 56, 43, 36, 26 ]\qquad\ 
\mathbf 1\mathbf 5: [ 120, 64, 39, 31, 25 ]\\
\mathbf 1\mathbf 6&: [ 120, 64, 56, 42, 28, 26 ]\\ 
\mathbf 1\mathbf 7&: [ 120, 64, 63, 43, 37, 36, 29, 
{\mathbf 2\mathbf 4}]
\qquad 
\mathbf 1\mathbf 8: [ 120, 56, 42, 39,  26,\mathbf 2\mathbf 3 ]\\
\mathbf 1 \mathbf 9&: [ 120, 64, 63, 42, 31, 28,\mathbf2\mathbf 4]
\qquad\quad\ \, 
\mathbf 2\mathbf 0: [ 120, 63, 56, 36, 31, 28,\mathbf 2\mathbf 0 ]
\\ 
\mathbf 2\mathbf 1&:[ 120, 64, 39, 37, 29,  25,\mathbf 2\mathbf 3 ]. 
\end{align*}
We see that for $2d=34$, $36$, $38$, $40$, $42$ there exits  only one class  $h_{2d}^\perp$ containing less than $50$ roots.
\end{proof}
We mentioned above that
the reflections defined by  $-2$-vectors in the lattice
$h^\perp_{U\oplus E_8(-1)}$ determine the transvections 
in the finite 
orthogonal group 
$\Orth^+(\mathbb F^{10}_2)\cong
\Tilde \Orth^+(N)\setminus \Orth^+(N)$. 
The table from the proof of the last corollary shows that
the group 
$\pi_M(\Orth(M,h))=\Tilde \Orth^+(N)\setminus \Gamma^+_h$
is large for small degrees.  
The next  interesting question is how small this group
might be.
\begin{proposition}\label{proposition-c0} There exist $h_{2d}$ such that 
$\Gamma^+_{h_{2d}}=\Tilde \Orth^+(N)$. 
\end{proposition}
\begin{proof} This proposition follows  form the following well-known 
fact in the  theory of quadratic forms:
most classes in  a genus of a positive definite quadratic form with large determinant have trivial orthogonal group.
More exactly,  let $c(\gothG)$ be the number of classes in the 
genus $\gothG$ and 
$c_0(\gothG)$ the number of classes $[L]\in \gothG$
such that $\Orth(L)=\{\pm 1\}$. Then 
\begin{equation*}
\frac{c_0(\gothG)}{c(\gothG)}\to 1, \qquad{\rm if}\quad
\det(\gothG)\to \infty.
\end{equation*}
This was proved by J. Biermann (1981) (see
\cite{Sch}). Therefore for a large $d$ there exists
a negative definite lattice $L\in \gothG(2d)$ such that 
$\Orth(L)=\{\pm 1\}$.
Taking a unimodular extension of $\latt{2d}\oplus L$ 
we obtain a primitive vector $h_{2d}$ in $U\oplus E_8(-1)$
such that $h^\perp_{2d}\cong L$. Then we get
$\Orth(M,h)\cong \Tilde \Orth(h^\perp_{2d})=\{1\}$ and 
\begin{equation*}
\Gamma_{h_{2d}}^+=\pi_N^{-1}(\pi_M(\Orth(M,h_{2d})))
\cap \Orth^+(N)=
\Tilde \Orth^+(N)
\end{equation*}
which clearly implies the claim.
\end{proof}
\begin{remark}\label{remark:gen-type}
We note that Proposition \ref{proposition-c0}
together with the fact mentioned at the end of \S 2  that the modular variety 
$\widetilde {\cM}_{\En}= \widetilde \Orth^+(N) \backslash \cD_N$
is of general type (see \cite{Gri2})
show that there exist moduli spaces  of numerically polarised 
(or polarized) Enriques surfaces of general type.
\end{remark}

\noindent
\begin{minipage}{0.5\textwidth}
Valery Gritsenko
\smallskip

\noindent
Laboratoire Paul Painlev\'e et IUF

\noindent
Universit\'e Lille 1

\noindent
F-59655 Villeneuve d'Ascq, Cedex

\noindent
France

\noindent
{\tt gritsenk@math.univ-lille1.fr}
\end{minipage}
\hfill
\begin{minipage}{0.4\textwidth}
{\it Current Address}:
\smallskip

\noindent
National Research University 

\noindent
Higher School of Economics

\noindent
AG Laboratory, HSE 

\noindent
7 Vavilova str.

\noindent 
Moscow, Russia, 117312
\end{minipage}

\bigskip
\medskip
\noindent
\begin{minipage}{0.5\textwidth}
Klaus Hulek
\smallskip

\noindent
Institut f\"ur Algebraische Geometrie

\noindent
Leibniz Universit\"at Hannover

\noindent
D-30060 Hannover

\noindent
Germany

\noindent
{\tt hulek@math.uni-hannover.de}
\end{minipage}
\hfill
\begin{minipage}{0.4\textwidth}
{\it Current Address}:
\smallskip

Insitute for Advanced Study

1, Einstein Drive

Princeton, NJ 08540 

USA

{\tt hulek@ias.edu}
\end{minipage}


\begin{thebibliography}{AMRT}

\bibitem[Bor]{Bor} R.~E.~Borcherds, {\it The moduli space of Enriques
  surfaces and the fake monster Lie superalgebra.} Topology {\bf 35}
  (1996), 699--710.

\bibitem[BHPV]{BHPV} W.~Barth, K.~Hulek, C.~Peters, A.~Van de Ven, 
  {\it Compact complex surfaces}. Second Enlarged Edition, Springer Verlag 2004.
  
\bibitem[Bea]{Beau}
 A.~I.~Beauville, J.-P.~Bourguignon, M.~Demazure (eds),
    {\it G\'eom\'etrie des surfaces K3: modules et p\'eriodes}.
    S\'eminaires Palaiseau, Ast\'erisque {\bf 126} (1985).
  
\bibitem[Bor]{Bo} A.~Borel, {\it Some metric properties of arithmetic
  quotients of symmetric spaces and an extension theorem.} J. 
  Differential Geometry {\bf 6} (1972), 543--560.
  
\bibitem[Cas]{Cas} G.~Casnati, {\it The moduli space of Enriques surfaces with a polarization of degree 4 is rational}. 
Geom. Dedicata {\bf 106} (2004), 185--194. 

\bibitem[CD]{CD} F.~Cossec, I.~Dolgachev, {\it Enriques Surfaces I}.
Progress in Mathematics, vol. {\bf 76}, 1989, Birkh\"auser.
 
 \bibitem[Die]{Die} J.~Dieudonn\'e, {\it La g\'eom\'etrie des groupes classiques (2nd ed.)}. Springer 1963.
 
\bibitem[Dol1]{Dol1}  I.~Dolgachev, {\it Rationality of fields of invariants}. In: S. J. Bloch (ed.), Algebraic Geometry (Bowdoin 1985), Proc. Sympos. Pure Math. {\bf 46}, 1987.

\bibitem[Dol2]{Dol2}  I.~Dolgachev, {\it Enriques surfaces: what is left?} 
Problems in the theory of surfaces and their classification (Cortona, 1988), Sympos. Math., {\bf XXXII},  81--97, Academic Press, London, 1991.

\bibitem[Dol3]{Dol3}  I.~Dolgachev, {\it A brief introduction to Enriques surfaces.} {\tt arXiv:1412.7744}, 31 pp.


\bibitem[Gri1]{Gri1}  V.~Gritsenko, \emph{Reflective modular forms 
in algebraic geometry}.  {\tt arXiv:1005.3753}, 28 pp.

\bibitem[Gri2]{Gri2}  V.~Gritsenko, Talk at the Schiermonnikoog conference, April 2014.

\bibitem[GH1]{GH1} V.~Gritsenko, K.~Hulek, {\it The modular form of
  the Barth--Nieto quintic.}  Intern. Math. Res. Notices {\bf 17}
  (1999), 915--938.

\bibitem[GH2]{GH2} V.~Gritsenko, K.~Hulek, {\it
Uniruledness of orthogonal modular varieties.}
J. Algebraic Geometry {\bf 23} (2014), 711--725.
 
\bibitem[GHS1]{GHS1} V.~Gritsenko, K.~Hulek, G.K.~Sankaran, {\it The
  Kodaira dimension of the moduli of K3 surfaces}. Invent. Math.
  {\bf 169} (2007), 519--567.  
  
\bibitem[GHS2]{GHS2} V.~Gritsenko, K.~Hulek, G.K.~Sankaran,
  \emph{Moduli spaces of irreducible symplectic manifolds.}
  Compos. Math. {\bf 146} (2010), 404--434. 
  
\bibitem[GHS3]{GHS3} V.~Gritsenko, K.~Hulek, G.K.~Sankaran, 
{\it Moduli spaces of K3 surfaces and holomorphic symplectic 
varieties.}
Handbook of Moduli(ed. G. Farkas and I. Morrison),  
vol. 1, 459--526; Adv. Lect. in Math, Intern. Press, MA, 2012. 
 
\bibitem[GN]{GN} V.~Gritsenko, V.~Nikulin, 
{\it Automorphic forms and Lorentzian Kac-Moody algebras, II.} 
International J. Math. {\bf 9}  (1998), 201--275. 
  
\bibitem[Has]{Has}  B.~Hassett, {\it Special cubic fourfolds.} Compositio Math. {\bf 120} (2000), 1--23.
  
\bibitem[Hor]{Hor} E.~Horikawa,  {\it On the periods of Enriques surfaces, I and II}. 
Math. Ann. {\bf 234} (1978), 73--88;
{\bf 235} (1978), 217--246. 

\bibitem[HP]{HP} K.~Hulek, D.~Ploog, {\it Fourier-Mukai partners and polarised K3 surfaces}.  In: Arithmetic and geometry of $K3$ surfaces and Calabi-Yau threefolds, 333--365, 
Fields Inst. Commun., {\bf 67}, Springer, New York, 2013. 

\bibitem[Kon1]{Kon1} S.~Kond\=o, {\it The rationality of the moduli space of Enriques surfaces}. Compositio Math. {\bf 91} (1994), 159--173. 

\bibitem[Kon2]{Kon2} S.~Kond\=o, {\it The moduli space of Enriques surfaces and Borcherds products}. J. Algebraic Geom. {\bf 11} (2002), 601--627. 

\bibitem[Lie]{Lie} Ch.~Liedtke, {\it Arithmetic moduli and lfting of Enriques surfaces}. J. reine angew. Math., DOI 10.1515/crelle-2013-0068.

\bibitem[Nam]{Nam} Y.~Namikawa,  {\it Periods of Enriques surfaces.}
Math. Ann {\bf 270} (1985), 201--222.
 
\bibitem[Nik]{Nik} V.V.~Nikulin, {\it Integral symmetric bilinear
forms and some of their applications.}  Izv. Akad. Nauk SSSR Ser.
Mat. {\bf 43} (1979), 111--177.  
English translation in Math. USSR,
Izvestiia {\bf 14} (1980), 103--167.

\bibitem[Sch]{Sch} R.~Scharlau,
{\it Kartin Kneser's work on quadratic forms and algebraic groups.}
In Contemporary Mathematics {\bf 493} 
``Quadratic Forms-Algebra, Arithmetic, and Geometry. 2009,
339--358.

\bibitem[Ste]{Ste} H.~Sterk, {\it Compactifications of the period
space of Enriques surfaces, I.}
Math. Z. {\bf 207} (1991), 1--36. 

\bibitem[Vie]{Vie} E.~Viehweg, {\it Quasi-projective moduli for polarized manifolds}. Ergebnisse der Mathematik und ihrer Grenzgebiete (3), vol. {\bf 30}, 
Springer, Berlin, 1995. 

\end{thebibliography}
\end{document}